\documentclass[10pt]{article}
\usepackage{graphicx}
\usepackage{indentfirst,csquotes}

\usepackage{amssymb,amsthm,amsmath}
\usepackage{xcolor,paralist,hyperref,titlesec,fancyhdr,etoolbox}
\newtheorem{theorem}{Theorem}[]
\newtheorem{definition}[theorem]{Definition}
\newtheorem{example}[theorem]{Example}
\newtheorem{lemma}[theorem]{Lemma}
\newtheorem{proposition}[theorem]{Proposition}
\newtheorem{corollary}[theorem]{Corollary}
\newtheorem{observation}[theorem]{Observation}
\newtheorem{remark}[theorem]{Remark}
\usepackage{lipsum}

\begin{document}
\title{On \(\mathbb{F}_q\)-Order of Polynomials and Properties of \(r\)-Primitive and \(k\)-Normal Elements over Finite Fields} 
\author{Maithri K., Vadiraja Bhatta G. R.$^*$, Indira K. P., Prasanna Poojary}
\date{}
\maketitle
\vspace{-1cm}
\begin{center}
    \small{Manipal Institute of Technology, MAHE, Manipal, Karnataka, India}
\end{center}

\begin{abstract}
Polynomials and elements over finite fields exhibit closely related algebraic structures, and many properties defined for elements extend naturally to polynomials. The concepts of order and $\mathbb{F}_q$-Order for elements have been extensively studied. In this paper, we investigate several properties of $r$-primitive and $k$-normal elements. Furthermore, by using the concept of the $\mathbb{F}_q$-Order of a polynomial, we explore properties of $k$-normal polynomials.
\end{abstract}

\bigskip

\section{Introduction}\label{sec1}

    ~~~Let \( q \) be a prime power and let \( n \) be a positive integer. Then \( \mathbb{F}_{q^n} \) denotes the finite field with \( q^n \) elements, which is an extension of \( \mathbb{F}_q \) of degree \( n \). This field exhibits two fundamental algebraic structures: it forms an $n$-dimensional vector space over $\mathbb{F}_q$, and its multiplicative group $\mathbb{F}_{q^n}^*$ is cyclic of order $q^n - 1$. These structures give rise to two distinguished classes of elements: normal elements and primitive elements.

    An element $\alpha \in \mathbb{F}_{q^n}$ is called a \textit{normal element} over $\mathbb{F}_q$ if the set of its conjugates $\{ \alpha, \alpha^q, \alpha^{q^2}, \dots, \alpha^{q^{n-1}} \}$ forms a basis of $\mathbb{F}_{q^n}$ as a vector space over $\mathbb{F}_q$. An element is said to be \textit{primitive} if it has multiplicative order $q^n - 1$, i.e., if it generates the entire multiplicative group $\mathbb{F}_{q^n}^*$.

    The study of primitive and normal elements began with the early works of Ore~\cite{OreFiniteFields} and Carlitz~\cite{Carlitz1952Prim, Carlitz1954Sum, CarlitzPrimitiveRootssomeproblem}, who looked at their main properties and set the foundation for later research in finite fields.
    The combination of these two properties in a single element, known as a \textit{primitive normal element}, is of particular interest due to its applications in cryptography and coding theory. Consequently, numerous studies have focused on the existence and construction of primitive normal elements~\cite{LenstraPriminormal, Cohen_1999, cohen2003primitive, SHARMA1}.

    The notions of primitive and normal elements have been extended to \textit{$r$-primitive} and \textit{$k$-normal} elements, which have been the focus of extensive recent research because of their importance in cryptography and coding theory.  A significant amount of work has focused on the existence of such elements and the investigation of their algebraic properties~\cite{Mullen2013Existance, Aguirre}. More specialized studies include results on the existence of primitive $1$-normal elements~\cite{Lucas2}. Subsequent research has examined the existence of \emph{pairs} of $r$-primitive and $k$-normal elements within finite field extensions~\cite{Aguirre2_pairs_rk}, as well as their characterization via inverse-type relations~\cite{RANI2_inverse}. Arithmetic progressions consisting of primitive, $r$-primitive, and $k$-normal elements have also been explored in several works \cite{Aguirre3_AP_rk, Laishram_AP}. In addition, various authors have considered these elements under additional constraints, such as arithmetic progressions with prescribed norms, $r$-primitive $k$-normal elements with specified norm and trace, and arithmetic progressions of $r$-primitive elements having a fixed norm \cite{chatterjee2024AP_norm, RANI1_norm&trace, ChoudharySpecial}.

    Primitive and normal elements play a central role in cryptography, and primitive polynomials are widely used in generating pseudorandom numbers through feedback shift registers (FSRs). This illustrates the importance of studying both the nature of elements and their polynomials, particularly primitive and normal polynomials.  A primitive polynomial is defined in terms of the order of its root, and such a polynomial has order $q^n-1$, as discussed in~\cite{LIDL}. The $k$-normal polynomials, introduced by M. Alizadeh et al.~\cite{knpoly}, are described through their roots. It's natural to be curious whether the notion of $\mathbb{F}_q$-Order can be introduced for polynomials and to define $k$-normal polynomials. This concept was briefly mentioned in Ore’s work~\cite{OreFiniteFields}. In this paper, we provide a precise definition of the $\mathbb{F}_q$-Order of a polynomial, in analogy with the case of elements, and use it to study structural properties and obtain new results pertaining $k$-normal polynomials.

    The paper is structured as follows. Section~\ref{sec2} collects the required preliminaries and foundational results used throughout the paper. In Section~\ref{sec4}, we define the notion of the $\mathbb{F}_q$-Order of a polynomial and examine its main properties and structural features. Section~\ref{sec3} is devoted to the study of $r$-primitive and $k$-normal elements over finite fields and analyzes their algebraic properties.

\section{Preliminaries}\label{sec2}
    This section summarizes the basic notions from finite field theory required for our study. They establish the foundational framework and notation used throughout the paper.

    \subsection{Multiplicative order and $r$-primitive elements}

    For $\alpha \in \mathbb{F}_{q^n}^*$, the smallest positive integer $d$ such that $\alpha^d = 1$ is called the \emph{multiplicative order} of $\alpha$, written $ord(\alpha)$. An element with order $q^n-1$ is called \emph{primitive}. More generally, an element is called \emph{$r$-primitive} if     $ord(\alpha) = \frac{q^n-1}{r},$
    where $r$ divides $q^n-1$~\cite{Reis2019Variation}. The case $r=1$ gives primitive elements.

    \subsection{$\mathbb{F}_q$-Order and $k$-normal elements}

    Let $\sigma_q(\beta) = \beta^q$ denote the Frobenius automorphism of $\mathbb{F}_{q^n}/\mathbb{F}_q$. For $\alpha \in \mathbb{F}_{q^n}$, the \emph{$\mathbb{F}_q$-Order} of $\alpha$, denoted $Ord(\alpha)$, is the unique monic polynomial $g(x) \in \mathbb{F}_q[x]$ of least degree such that $g(\sigma_q)(\alpha) = 0$. Equivalently, it is the minimal polynomial over $\mathbb{F}_q$ that annihilates $\alpha$ under the Frobenius action. It always divides $x^n-1$.

    An element $\alpha \in \mathbb{F}_{q^n}$ is called \emph{normal} over $\mathbb{F}_q$ if the set $\{ \alpha, \alpha^q, \dots, \alpha^{q^{n-1}} \}$ forms a basis of $\mathbb{F}_{q^n}$ over $\mathbb{F}_q$. Mullen et al.~\cite{Mullen2013Existance} introduced the concept of \emph{$k$-normal elements} which is defined by the condition
    $\deg(Ord(\alpha)) = n-k,$ 
    which is equivalent to
    \[\dim_{\mathbb{F}_q}\langle \alpha, \alpha^q, \dots, \alpha^{q^{n-1}} \rangle = n-k.\]
    The case $k=0$ corresponds to normal elements.

    \subsection{$m$-free and $g$-free elements}

    Additionally, we look into $r$-primitive and $k$-normal elements using the additional characterization of finite field elements known as freeness. 
    
    \begin{definition}
        \begin{enumerate}
          \item For $m \mid q^n-1, \alpha \in \mathbb{F}_{a^n}^*$ is $m$-free if $\alpha=\beta^d$, for any divisor $d$ of $m$, implies $d=1$.
          \item For $M \mid x^n-1, \alpha \in \mathbb{F}_{q^n}$ is $M$-free if $\alpha=H(\beta)$, where $H$ is the $q$-associate of a divisor $h$ of $M$, implies $h=1$\cite{Mullen2013Existance}.
      \end{enumerate}
    \end{definition}

    \medskip
    \section{$\mathbb{F}_q$-Order of a Polynomial}\label{sec4} 
    For elements of finite fields, the multiplicative order and the
$\mathbb{F}_q$-Order display closely related behavior; in particular,
both notions are governed by analogous structural principles. A similar
phenomenon occurs for polynomials over finite fields. The concept of the
order of a polynomial over $\mathbb{F}_{q^n}$ is well established: the
order of a polynomial $f(x)$ is defined as the least positive integer
$e$ such that
\[
f(x) \mid (x^e - 1).
\]

Motivated by this analogy, we introduce the notion of the
$\mathbb{F}_q$-Order of a polynomial and investigate its fundamental
properties. We show that, for polynomials over finite fields, the notions
of order and $\mathbb{F}_q$-Order exhibit parallel structural behavior,
closely mirroring the relationship observed in the case of field
elements.

The origin of this idea can be traced back to the work of Ore~\cite{OreFiniteFields},
who introduced polynomials belonging to the class of $p$-polynomials.
However, a systematic and detailed study of this concept has remained
limited. In this work, we develop a more formal treatment within the
modern framework of finite fields, interpreting this notion as the
$\mathbb{F}_q$-Order of a polynomial, with a suitable refinement of the
classical definition. 
\begin{definition}
    Let $f \in \mathbb{F}_q[x]$ be a polynomial of degree $m \geq 1$. Then $g$ is said to be $\mathbb{F}_q$-Order of $f$ if $g$ is the monic least degree polynomial such that $f \vert g\circ x$. We  denote it by $Ord(f)$.
\end{definition}

\begin{example}
        Consider the field $\mathbb{F}_{5^2}$. Let $f=x+1$. Then $\mathbb{F}_q$-Order of $f$ is $x-1$, since it is the least degree polynomial such that $f$ divides $x^5 - x$.
\end{example}

When considering an irreducible polynomial, the following result shows the relationship between the polynomial's $\mathbb{F}_q$-Order and its root. They will actually have equal $\mathbb{F}_q$-Orders.
    \begin{proposition}\label{fequala}
        If $f \in \mathbb{F}_q[x] $ is irreducible polynomial over $\mathbb{F}_q$ of degree $m$, then $Ord(f)=Ord(\alpha)$, where $\alpha $ is the root of $f$.
    \end{proposition}
    \begin{proof}
        Let $Ord(f) = g$, then by definition $f \vert g \circ x$, which implies $g \circ x =f(x).h(x)$, for some polynomial $h(x)$. Since $\alpha$ is the root of $f$, $g \circ \alpha = 0$.\\
        Let $k$ be a polynomial in $\mathbb{F}_q[x]$, with degree less than degree of $g$ and $h \circ \alpha =0$. Which implies that $h \circ \alpha ^ {q^i}=0, \forall~i =1, 2,\dots,m$, from which we get $f \vert h \circ x$, a contradiction. Thus $\mathbb{F}_q$-orders of $\alpha=g$.
    \end{proof}
    \begin{corollary}\label{fdivx}
         If $f \in \mathbb{F}_q[x] $ is irreducible polynomial over $\mathbb{F}_q$ of degree $m$, then $Ord(f) \vert x^m -1$.
    \end{corollary}

    \begin{proposition}\label{1}
        Let $f \in \mathbb{F}_q[x] $ be a polynomial such that $f \vert g\circ x$ for some $g \in \mathbb{F}_q[x]$ if and only if $Ord(f) \vert g$.
    \end{proposition}
    \begin{proof}
        Let $h= Ord(f) \vert g$, then $f \vert h \circ x$. Since $h\circ x \vert g \circ x$, we get $f \vert g \circ x$.\\
        Conversly, let $f \vert g \circ x$, then $deg(g) \geq deg(h)$. Thus $g = hk + r$, for some $ k, r \in \mathbb{F}_q[x]$ with $0 \leq deg(r) < deg(h)$. Now, $g \circ x = (hk) \circ x + r \circ x$. Also, $f \vert g \circ x$ and $f \vert (hk) \circ x$. Thus $f \vert r \circ x$, which implies $r \circ x = 0$, which implies $r=0$. Thus $h \vert g$.
    \end{proof}
    \begin{corollary}
        If $h, g$ are two non-constant polynomials. Then $gcd(h\circ x, g \circ x) = k \circ x$, where $k = gcd(h, g)$.
    \end{corollary}
    
    Now we consider irreducible polynomials of various type and see what happens to their $\mathbb{F}_q$-Order.
    \begin{proposition}\label{7}
        Let $g \in \mathbb{F}_q[x]$ be irreducible over $\mathbb{F}_q$ and $f=g^b$ for some positive integer $b$. Let $t$ be the smallest positive integer with $p^t \geq b$, where p is the characteristic of $\mathbb{F}_q$. Then $Ord(f) = Ord(g)x^t$.
    \end{proposition}
    \begin{proof}
        Let $Ord(g) = h$ and $Ord(f) = k$. Then, $g \vert h \circ x$, which implies $ g^b \vert (h\circ x)^b$,  $g^b \vert (h \circ x)^{q^t} $ and then $ g^b \vert h \circ x^{q^t} $. Using $g^b = f$, we get $ f \vert h \circ x^{q^t}$. Thus, $Ord(f) \vert hx^t$. \\
        Let $deg(h) = s$. We have, $g \vert h \circ x, $ which implies, $ g^b \vert (h\circ x)^b$, where $deg((h\circ x)^b) = q^sb$.
        Let $\alpha_1,\cdots, \alpha_m$ be distinct roots of $g$. We know that the roots of $h$ forms a subspace U. Since $h$ is the $\mathbb{F}_q$-Order of $g$, U is the smallest subspace containing $\alpha_1, \cdots, \alpha_m$. 
        Let $L$ be the least degree linearized polynomial such that $g^b \vert L$ (such a polynomial exits as $g^b \vert (h \circ x)^{q^t}$ ). Then $L$ should contain the factor $(x-\alpha_i)^{q^t}$.
        Since roots of $L$ forms a subspace $V$ and $L$ is the least degree linearized polynomial with $g^b\vert L$, $V=U$ which gives $q^sb \leq deg(L)$. By the choice of $t$, $q^sq^t \leq deg(L)$. Also, from proposition \ref{1},
        we have $L \vert (hx^t) \circ x$. Thus, $q^sq^t \geq deg(L) \implies L=(hx^t) \circ x$. Thus $Ord(f) = Ord(g)x^t$.  
    \end{proof}
    \begin{proposition}\label{8}
        Let $g_1,g_2,\dots, g_k$ be pairwise relatively prime polynomials over $\mathbb{F}_q[x]$, and $f=g_1g_2 \dots g_k$. Then $Ord(f) = lcm\{Ord(g_1), Ord(g_2), \dots , Ord(g_k)\}$.
    \end{proposition}
    \begin{proof}
        Let $Ord(g_i) = h_i ~ \forall~i, 1 \leq i \leq k$ and $l = lcm \{ h_1, h_2, \dots h_k \} $, then $h_i \vert l ~\forall~i$, further $H_i\vert l \circ x ~\forall~i$, where $H_i= h_i \circ x$.\\
        We have, $g_i \vert H_i ~ \forall~i$, which implies that $g_1g_2\dots g_k \vert lcm\{H_1,H_2,\dots,H_k\} \vert l \circ x$. Thus, $f \vert l \circ x$, which implies $Ord(f) \vert l$.\\
        Let $Ord(f) = j$, then $f \vert j \circ x$, which implies $g_i \vert j \circ x ~ \forall~i$. Since $Ord(g_i) = h_i, h_i \vert j ~\forall~i$. which implies $lcm\{h_1,h_2, \dots, h_k\} \vert j=Ord(f)$.     
    \end{proof}
    If the $\mathbb{F}_q$-Order of an element in $\mathbb{F}_{q^n}$ is $x^n-1$, then the element is normal. The identical thing holds true for normal polynomials, which is defined as the minimal polynomial of a normal element in a finite field \cite{LIDL}.
    \begin{proposition}
        A polynomial $f \in \mathbb{F}_q[x]$ of degree m is normal over $\mathbb{F}_q$ if and only if $f$ is monic and $Ord(f) = x^m-1$.
    \end{proposition}
          \begin{proof}
        Let $f$ be a normal polynomial in $\mathbb{F}_q[x]$, and $g=Ord(f)$. Which implies $f \vert g \circ x$.\\
        Let $\alpha$ be the root of $f$, then $Ord(\alpha) = x^m-1$. Since $f \vert g \circ x, ~ g\circ \alpha = 0 \implies g= Ord(f) = x^m - 1$.\\
        Conversely, Let $Ord(f) = x^m - 1$. Suppose $f$ is reducible, then either $f=g^b$, for some irreducible polynomial $g$, or $f=gh$ with $g,h$ irreducible or $f=g^ah^b$ for some relatively prime polynomials $g,h$ with $a,b\geq 1$\\
        Suppose $f=g^b$, then by proposition \ref{7}, $x \vert Ord(f) \implies x\vert x^m-1$ a contradiction.\\
        Suppose $f=g^ah^b$, with $g^a$ and $h^b$ being relatively prime and $a \text{ or } b\geq 1$ then $Ord(f) = lcm(Ord(g^a),Ord(h^b))$, a multiple of $x$ from proposition \ref{8} which implies $x \vert x^m-1$, a contradiction.\\
        Suppose $f=gh$ with $f,g$ being irreducible, $deg(g)=e_1$ and $deg(h) = e_2$ then $m=e_1+e_2$.We know that $x^m-1=Ord(f)=lcm(Ord(g), Ord(h))$.
        Now by corollary \ref{fdivx} $Ord(g) \vert x^{e_1} - 1, Ord(h) \vert x^{e_2} - 1$. So $lcm(Ord(g),Ord(h))\vert lcm(x^e_1-1,x^e_2-1)$ which implies $ x^m-1\vert lcm(x^e_1-1,x^e_2-1) $ which gives, $ m< e_1+e_2$ a contradiction.\\
        Therefore $f$ is a irreducible polynomial over $F_q$. So by preposition \ref{fequala}, we have $F_q$-Order of all of its roots is $x^m-1$. Hence the roots are normal elements. Hence $f$ is normal polynomial.        
    \end{proof}
    We know number of elements with specific $\mathbb{F}_q$-Order. Now we derive a formula for number of irreducible polynomials with fixed degree whose $\mathbb{F}_q$-Order is same. 
    \begin{lemma}\label{degmin}
        If $\alpha$ is an element with $\mathbb{F}_q$-Order $f$, then degree of minimal polynomial of $\alpha$, $deg(m_{\alpha}(x)) = min\{ r \geq 1 : f(x) \vert x^r - 1 \} = ord(f)$.
    \end{lemma}
    \begin{proof}
        Consider the Frobenius automorphism $\sigma$, which maps $\alpha$ to $\alpha^q$ for every $\alpha$. Since $f(x)$ is the $\mathbb{F}_q$-Order of $\alpha$, $(f(x)) = \{ g(x) : g(\sigma)(\alpha) = 0\} $. The condition $\alpha^{q^r} = \alpha$ is equivalent to $(x^r - 1)(\sigma)(\alpha) = 0$, which is if and only if $x^r - 1 \in (f(x))$ or equivalently $f \vert x^r-1$. \\ 
        Now, using the result for degree of minimal polynomial in Chapter 19 of \cite{Shoup2009}, degree of $m_{\alpha} (x) = min\{r>0: \alpha^{q^r} = \alpha\}$, by previous discussions which is same as $min\{ r \geq 1 : f(x) \vert x^r - 1 \} = ord(f)$.
    \end{proof}
    \begin{proposition}
        The number of irreducible polynomials of degree n with $\mathbb{F}_q$-Order $f$ is $ \dfrac{\phi(f)}{n}$ only when $ord(f)$ is $n$, 0 otherwise.
    \end{proposition}
    \begin{proof}
        \textbf{Case 1:} $ord(f) = n$\\
        Let $g_1, g_2, \dots, g_k$ be the irreducible polynomials with degree n whose $\mathbb{F}_q$-Order is $f$, and $S$ be the set of all roots of them. Then $|S| = nk$. Also, $B$ be the set of all elements in $\mathbb{F}_{q^n}$ with $\mathbb{F}_q$-Order $f$. From proposition $\ref{fequala}$, $S \subseteq T$.

        From lemma \ref{degmin}, all the elements of $T$ have minimal polynomial of equal degree. Since $S \subseteq T, S \neq \phi$, any $\alpha \in S$ has minimal polynomial of degree n i.e one of the $g_i$'s by the definition. Thus all elements of $T$ has minimal polynomial of degree $n$. So it should be one of the $g_i$'s. Thus $S = T$. Therefore number of irreducible polynomials of degree $n$ with $\mathbb{F}_q$-Order $f$ is given by $\dfrac{\phi(f)}{n}$ whenever $n \vert \phi(f)$. \\
        \textbf{Case 2:} $ord(f) \neq n$\\
        If order of $f$ is other than $n$, then by lemma \ref{degmin}, degree of minimal polynomial of elements with $\mathbb{F}_q$-Order $f$ is not equal to $n$. Thus, they are not among the $g_i$'s. Thus number will be equal to zero.
    \end{proof}
    The $k$-normal polynomials are defined by Mahmood Alizadeh et. al \cite{knpoly}. A monic irreducible polynomial of degree $n$ over a finite field $\mathbb{F}_q$ is said to be $k$-normal if its roots are the $k$-normal elements over the field $\mathbb{F}_{q^n}$. Using the above results number of $k$-normal polynomials of degree $n$ will be equal to $ \displaystyle \sum_{\substack{h \vert x^n-1 \\ deg(h) = n-k \\ n \vert \phi(h)}}\frac{\phi(h)}{n}$.
    
\begin{theorem}
    Let $f\in \mathbb{F}_q[x]$ with $deg(g)\geq 1.$ Let $a \in \mathbb{F}_q$ be such that $h$ be the least degree polynomial with $f(x)\vert (h\circ x-a)$. Then $Ord(f)=Ord(a).h$
\end{theorem}
\begin{proof}
    Let $k=Ord(f)$, then $f(x)\vert k\circ x$ which implies $deg(k)\geq deg(h)$.\\
     Therefore by division algorithm $k=hm+n$ with $0\leq deg(n) < deg(h)$ or $n=0$ which implies $$ (hm+n)\circ x\equiv 0~ mod{f(x)}$$ which gives $$ n\circ x \equiv -m\circ a~ mod{f(x)}$$ as  $(hm)\circ x=m\circ h \circ x \equiv m\circ a~mod{f(x)}$.\\ Hence $n=0$ which gives $Ord(a)\vert m$ as $m\circ a=0$. Therefore $(Ord(a)h)\vert hm=k$. Conversly, $$(Ord(a)h)\circ x \equiv Ord(a) \circ a\equiv 0 ~mod{f(x)}$$
     Hence $k \vert Ord(a)h$.
\end{proof}
\section{Algebraic and Combinatorial Properties of $r$-Primitive and $k$-Normal Elements}\label{sec3}
    The following lemmas help in constructing r-primitive and k-normal elements in higher fields.
    \begin{lemma}
        Let $m\vert n$ and $\alpha \in \mathbb{F}_{q^m}, \text{ then } \alpha$ is $r$-primitive in $\mathbb{F}_{q^m}$ if and only if $\alpha$ is $kr$-primitive in $\mathbb{F}_{q^n}$, where $k=\dfrac{q^n-1}{q^m-1}.$
    \end{lemma}
    \begin{lemma}
        Let $m\vert n$ and $\alpha \in \mathbb{F}_{q^m}, \text{ then } $ $\alpha$ is $k$-normal in $\mathbb{F}_{q^m}$ if and only if $\alpha$ is $n-m+k$-normal in $\mathbb{F}_{q^n}$.
    \end{lemma} 
    In a cyclic field, $ord(ab) = lcm(ord(a),ord(b))$ where the order of $a$ and $b$ are relatively prime. Using this concept, we obtain the following result.
    \begin{proposition}\label{Ordmul}
        Let $\mathbb{F}_{q^n}$ be a field and $a$ be $r_1$-primitive, $b$ be $r_2$-primitive in $\mathbb{F}_{q^n}$. If $gcd(\dfrac{q^n-1}{r_1}, \dfrac{q^n-1}{r_2}) = 1$, then $ab$ is $gcd(r_1, r_2)$-primitive in $\mathbb{F}_{q^n}$.
    \end{proposition}
    \begin{proof}
        $a$ be $r_1$-primitive, $b$ be $r_2$-primitive in $\mathbb{F}_{q^n}$\\
        Then, $ord(a) = \dfrac{q^n-1}{r_1}$ and $ord(b) = \dfrac{q^n-1}{r_2}.$\\
        Since $gcd(ord(a), ord(b)) = 1, 
         ord(ab) = lcm(ord(a),ord(b)) = lcm(\dfrac{q^n-1}{r_1}, \dfrac{q^n-1}{r_2})=\dfrac{q^n-1}{gcd(r_1, r_2)}$ which implies that the element $ab$ is $gcd(r_1, r_2)$-primitive in $\mathbb{F}_{q^n}$.
    \end{proof}

    It is possible to identify a partition of the set of all $k$-normal elements \cite{RANI2_inverse}. In contrast, $t$-primitive elements, where $t$ is a multiple of $r$, can be used to partition the set of elements that are $r$-th powers in the case of $r$-primitive elements. An alternative proof for $\sum _{d \vert n} \phi(d) = n$ can be provided using the findings below.

    \begin{theorem}\label{Th1}
        For any divisor $r$ of $q^n-1$, let $\mathrm{B} = \{ \beta ^r ~\vert ~ \beta \in \mathbb{F}_{q^n}\}$ and $\mathrm{A_t} = \{\alpha \in \mathbb{F}_{q^n} ~\vert ~ord(\alpha) = \frac{q^n-1}{t}\} $, where $t$ is a multiple of $r$. Then \[\mathrm{B} = \bigcup_{r\vert t} \mathrm{A_t}.\]
    \end{theorem}
    \begin{proof}
        $\text{Let }a \in  \displaystyle\bigcup_{r\vert t}\mathrm{A_t},\text{ then } a \in \mathrm{A_t} \text{ for some multiple } t \text{ of } r.$
        According to the definition of $\mathrm{A_t}$, $ord(a)$ = $\dfrac{q^n-1}{t}$ where $t \vert q^n-1$ and $t = kr$ for some $k \in \mathbb{N}$. Since $a$ is $t$-primitive, $a = \gamma ^ t$, where $\gamma$ is a primitive element. 
        i.e. $a = \gamma ^ {kr} = (\gamma ^k)^r = \beta ^r \text{, where } \beta = \gamma ^k \in \mathbb{F}_{q^n}.$ Thus, $a\in B.$\\
        Let $a\in \mathrm{B},$ then $a= \beta ^r$ for some $r$.\\
        Let $a$ be $t$-primitive for some $t \vert q^n - 1$. \\
        Consider, $a^{\frac{q^n-1}{r}} = (\beta ^r)^{\frac{q^n-1}{r}} = 1.$ Since the $ord(a)$ is $\dfrac{q^n-1}{t}$ we get, $\dfrac{q^n-1}{t}\vert \dfrac{q^n-1}{r} $ which shows that $t$ is a multiple of $r$. Thus
        $ a \in A_t$, which implies $ a \in  \displaystyle\bigcup_{r\vert t}\mathrm{A_t}$.\\
          $\text{Thus, }\mathrm{B} =  \displaystyle\bigcup_{r\vert t} \mathrm{A_t}.$
    \end{proof}
    The number of elements in the different sets defined in the theorem \ref{Th1} is provided by the next observation and the proposition.
    \begin{observation}
        \begin{enumerate}
            \item $ \displaystyle\bigcup_{r\vert t} \mathrm{A_t}$ forms a partition of $B.$
            \item $|A_t| = \phi(\frac{q^n - 1}{t})$, for any $t$.
            \item   $|B| =  \displaystyle\sum_{r\vert t} \phi(\frac{q^n - 1}{t})$ \label{obs3}
        \end{enumerate}
    \end{observation}
    \begin{proposition}
        Let $r\vert q^n - 1$ and $\mathrm{B} = \{ \beta ^r ~\vert ~ \beta \in \mathbb{F}^*_{q^n}\}$, then $|\mathrm{B}| = \dfrac{q^n-1}{r}.$
    \end{proposition}
    \begin{proof}
        Let $\sigma: \mathbb{F}^*_{q^n} \longrightarrow \mathbb{F}^*_{q^n}$ be defined by $\sigma(x) = x^r.$ \\
        Then, $\sigma$ is a homomorphism.\\
        Infact, for any $x,y \in \mathbb{F}^*_{q^n}, ~\sigma(xy) = (xy)^r = x^ry^r = \sigma(x)\sigma(y).$\\
        Now, $Ker~\sigma = \{ x \in \mathbb{F}^*_{q^n} ~|~  x^r = 1\}$ and $Im ~ \sigma = \{ x^r ~\vert ~ x \in \mathbb{F}^*_{q^n}\} = \mathrm{B}.$\\
        Since $r\vert q^n-1, x^r-1$ splits completely in $\mathbb{F}^*_{q^n}$ and 
        $x^r - 1$ has $r$ distinct roots in $\mathbb{F}^*_{q^n}$, as $x^r - 1$ and its derivative have no common roots.\\
        Thus, $|Ker~\sigma| = r$.\\
        By First Isomorphism Theorem, $\mathbb{F}^*_{q^n} / Ker~\sigma \cong B $, Thus, $ |B| = \dfrac{q^n-1}{r}.$ 
    \end{proof}    
    For any divisor $r$ of $q^n - 1$, let $\dfrac{q^n-1}{r} = k,$ $r \vert t $ is equivalent to $ l=\displaystyle\frac{q^n-1}{t} \mid \frac{q^n-1}{r}.$ Then the set $\mathrm{B}$ defined in previous proposition, have cardinality $ k.$ By observation \ref{obs3}, we get $|\mathrm{B}| =  \displaystyle\sum_{r\vert t} \phi(\frac{q^n - 1}{t}) = \frac{q^n-1}{r} $ which gives the proof of $ \displaystyle\sum_{l\vert k} \phi(l) = k$.

    The $\mathbb{F}_q$-Order also exhibits behavior analogous to that of classical orders. In particular, the set of all elements whose $\mathbb{F}_q$-Order is a multiple of $f$ induces a similar partition structure.
    \begin{proposition}
For any divisor $f$ of $x^{n}-1$, let
$$
B = \{\, a \in \mathbb{F}_{q^{n}}^{*} : f \mid \operatorname{Ord}(a) \,\},
$$
and for every multiple $g$ of $f$, define
$$
A_{g} = \left\{\, \alpha \in \mathbb{F}_{q^{n}}^{*} \ \middle|\ 
\operatorname{Ord}(\alpha)=\frac{x^{n}-1}{g} \right\}.
$$
Then
$$
B = \bigcup_{f \mid g} A_{g}.
$$
\end{proposition}

\begin{proof}
Suppose $\displaystyle a \in \bigcup_{f \mid g} A_{g}$.  
Then $a \in A_{g}$ for some multiple $g$ of $f$, and therefore $\operatorname{Ord}(a) = \dfrac{x^{n}-1}{g}.$

There exists a normal element $b \in \mathbb{F}_{q^{n}}^{*}$ such that $a = g(b)$.  
Since $g$ is a multiple of $f$, write $g = hf$ for some polynomial $h$.  
Thus
$$
a = h f(b) = f(h(b)) = f(c),
$$
where $c = h(b) \in \mathbb{F}_{q^{n}}^{*}$.  
Hence $a \in B$, giving
$$
\bigcup_{f \mid g} A_{g} \subseteq B.
$$

Conversely, assume $a \in B$.  
Then $a = f(b)$ for some $b \in \mathbb{F}_{q^{n}}^{*}$ and
$$
\operatorname{Ord}(a) = \frac{x^{n}-1}{g}.
$$

Now consider
$$
\frac{x^{n}-1}{f} \circ a
    = \frac{x^{n}-1}{f} \circ (f(b))=
    (x^{n}-1)\circ (b)
    = 0,
$$
Since $\operatorname{Ord}(a) = \dfrac{x^{n}-1}{g}$,
$$
\frac{x^{n}-1}{g} \mid \frac{x^{n}-1}{f}
    \quad\Longrightarrow\quad f \mid g.
$$

Hence $a \in A_{g}$, which implies
$$
B \subseteq \bigcup_{f \mid g} A_{g}.
$$
Thus, $B = \displaystyle \bigcup_{f \mid g} A_{g}$.
\end{proof}
\begin{proposition}
    Let $f \vert x^n-1$ and $B = \{\, a \in \mathbb{F}_{q^{n}}^{*} : f \mid \operatorname{Ord}(a) \,\}$, then $|B| = q^{n-deg f}$.
\end{proposition}
\begin{proof}
     Let $\sigma_f : \mathbb{F}_{q^{n}} \rightarrow \mathbb{F}_{q^{n}}$   defined by $\sigma_f(x)=f(x)$.
 Then $\sigma_f$ is a homomorphism. Infact, for any $a,b \in \mathbb{F}_{q^{n}}$,
\[
\sigma_f(a+b)=f(a+b)=f(a)+f(b)=\sigma_f(a)+\sigma_f(b).
\]

Consider  
$Ker \sigma_f = \{\alpha \in \mathbb{F}_{q^{n}} : f(\alpha)=0\}$,  
the set of all elements whose $\mathbb{F}_{q}$-Order divides $f$.

Thus,
\[
|Ker \sigma_f|
 = \sum_{h \mid f} \phi(h)\, q^{\deg(f)-\deg(h)}.
\]

The image of $\sigma_f$ is  
$\operatorname{Im}(\sigma_f)=\{f(\alpha):\alpha\in\mathbb{F}_{q^{n}}\}=B$.

By the first isomorphism theorem,
\[
\frac{\mathbb{F}_{q^{n}}}{\ker \sigma_f}
\cong \operatorname{Im}(\sigma_f)=B.
\]

Therefore,
\[
|B|=\frac{q^{n}}{|\ker\sigma_f|}
   =q^{\,n-\deg(f)}.
\]

\end{proof}

    Using the above results, one obtains an alternative proof of the identity
\[
\sum_{g \mid f} \phi(g) = q^{\deg(f)}.
\]

    The study of $r$-primitive and $k$-normal elements heavily relies on $m$-free and $g$-free elements respectively. The number of $m$-free elements can be found in the following comment. This serves as motivation for calculating the number of $g$-free elements.

    \begin{remark}
        Number of $m$-free elements over $\mathbb{F}_{q^n}$ is $\dfrac{\phi(m)(q^n-1)}{m}$.
    \end{remark}
    \begin{proposition}
        Number of $g$-free elements over a finite field $\mathbb{F}_{q^n}$ is $\dfrac{\phi(g)q^n}{q^{deg(g)}}$.
    \end{proposition}
    \begin{proof}
        Let $x^n-1 = (\displaystyle \prod_{i=1}^r f_i(x))^a$ and $g=\displaystyle \prod_{i=1}^k (f_i(x))^{b_i}$ where each $f_i$ is the irreducible polynomial of degree $n_i$ and $k \leq r$. It is known that $\alpha \in \mathbb{F}_{q^n}$ is $g$-free if and only if $\gcd \left( \dfrac{x^n-1}{Ord(\alpha)}, g \right) = 1$. Thus $Ord(\alpha)$ is of the form $f_1^af_2^a \dots f_k^a h$, where $h$ is a divisor of $f_{k+1}^a \dots f_r^a$.\\
        Thus, the number of $g$-free elements is equals to, 
        \begin{align*}
            &= \sum_{h \vert f_{k+1}^a \dots f_r^a} \phi(f_1^af_2^a \dots f_k^a h) \\
            &= \phi(f_1^af_2^a \dots f_k^a) \sum_{h \vert f_{k+1}^a \dots f_r^a} \phi(h) \\
            & = \phi(f_1^af_2^a \dots f_k^a) \sum_{i_{k+1} = 0}^a \sum_{i_{k+2} = 0}^a \dots \sum_{i_{r} = 0}^a  \phi(f_{k+1}^{i_{k+1}}f_{k+2}^{i_{k+2}}\dots f_r^{i_r})
        \end{align*}
        Further calculations will lead to give the number of $g$-free elements to be $\dfrac{\phi(g)q^n}{q^{deg(g)}}$.
        
    \end{proof}
    If $\alpha$ is $r$-primitive, then in any field $\mathbb{F}_{q^n}$, its inverse will also be $r$-primitive. However, this need not always occur when $k$-normal is involved. It happens only in the fields that are quadratic extensions or for $k = n-1$. 

    \begin{proposition}
        In any field $\mathbb{F}_{q^n}$, if $\alpha$ is $(n-1)$-normal, then $\alpha^{-1}$ is $(n-1)$-normal.
    \end{proposition}
    \begin{proof}
        Let $\alpha$ be $(n-1)$-normal, then $\mathbb{F}_q$ order of $\alpha$ is $f=x+a$ for some $a \in \mathbb{F}_q$. Which implies $(x+a) \circ \alpha = 0 \implies \alpha ^ q + a \alpha = 0 $. Thus, $ \alpha ^q = -a \alpha$.\\
        Now, consider $f^*(x) = x a^{-1} f(\frac{1}{x}) = x+a^{-1}$. Then, $f^* \circ \alpha^{-1} = (\alpha^{-1})^q+a^{-1}\alpha ^{-1} = -a^{-1}\alpha^{-1} = 0$. Thus, $\mathbb{F}_q$ order of $\alpha^{-1}$ is $f^*$, and hense it is $(n-1)$-normal.
    \end{proof}
    \begin{corollary}
        Let $\alpha \in \mathbb{F}_{q^2}$. If $\alpha$ is 0-normal then $\alpha^{-1}$ is also 0-normal and if $\alpha$ is 1-normal then $\alpha^{-1}$ is also 1-normal.
    \end{corollary}

    Theorem~\ref{Ordmul} provides a property concerning the product of \( r \)-primitive elements. In the context of \( \mathbb{F}_q \)-Order, we consider the sum over elements whose \( \mathbb{F}_q \)-Orders are relatively prime.

    \begin{proposition}
        Let $\alpha, \beta \in \mathbb{F}_{q^n}$ with Ord$(\alpha) = f$, Ord$(\beta) = g$ and gcd$(f,g)$ = 1. Then Ord$(\alpha + \beta) = fg$.
    \end{proposition}
        \begin{proof}
        Let $k=$Ord$(\alpha+\beta)$ and $f(x)=a_0+a_1x+a_2x^2+ \dots + a_mx^m, g(x) = b_0+b_1x+\dots+b_lx^l$. \\
        Consider, 
        \begin{align*}
            fg \circ (\alpha + \beta) &= f \circ (b_0(\alpha+\beta)+b_1(\alpha+\beta)^{q} + b_2(\alpha+\beta)^{q^2} + \dots + b_l(\alpha+\beta)^{q^l})\\
            &= f \circ ((b_0 \alpha + b_1 \alpha^q + \dots + b_l \alpha^{q^l}) + (b_0 \beta + b_1 \beta^q + \dots + b_l \beta^{q^l}))\\
            &= f \circ ((b_0 \alpha + b_1 \alpha^q + \dots + b_l \alpha^{q^l})+g\circ \beta)\\
            &= a_0(b_0 \alpha + b_1 \alpha^q + \dots + b_l \alpha^{q^l})+a_1(b_0 \alpha + b_1 \alpha^q + \dots + b_l \alpha^{q^l})^q+\dots \\
            & ~ ~ + a_m(b_0 \alpha + b_1 \alpha^q + \dots + b_l \alpha^{q^l})^{q^m}\\
            &= b_0(a_0\alpha + a_1 \alpha^q+\dots a_m \alpha ^{q^m}) + b_1(a_0\alpha^q+ a_1 \alpha^{q^2} + \dots a_m \alpha ^{q^{m+1}})+\dots \\
            & ~ ~ + b_n(a_0\alpha^{q^m}+a_1\alpha^{q^{m+1}}+\dots a_m \alpha^{q^{m+l}})\\
            &= b_0(f \circ \alpha)+b_1(f \circ \alpha ^q)+ \dots +b_n(f \circ \alpha^{q^l})\\
            &= 0.
        \end{align*}
        Thus $k \mid fg$.
        Now, $k \circ ( \alpha + \beta ) = k \circ \alpha + k \circ \beta = 0$, which implies $k \circ \alpha = - k \circ \beta = k \circ (-\beta)$.\\
        Then, $f \circ (k \circ \alpha) = f \circ (k \circ (-\beta))$, which implies $fk \circ \beta = 0$.
        Thus $ g  \mid  fk$, which implies $ g  \mid k$.\\
        Similarly, $ f  \mid  k$. Thus $ fg  \mid  k$.\\
        Therefore $k=fg$.
    \end{proof}
    \begin{corollary}
        Let $\alpha, \beta \in \mathbb{F}_{q^n}$ with Ord$(\alpha) = f$, Ord$(\beta) = g$ and gcd$(f,g)$ = 1. Let $\alpha$ be $k_1$-normal and $\beta$ be $k_2$-normal. Then $\alpha + \beta$ is $k_1+k_2-n$-normal.
    \end{corollary} 
    \begin{corollary}
        Let $\alpha_1, \alpha_2, \dots, \alpha_k \in \mathbb{F}_{q^n}$, with $Ord(\alpha_i)=g_i, ~\forall~i, 1 \leq i \leq k$ and $g_1,g_2,\dots, g_k$ be pairwise relatively prime polynomials over $\mathbb{F}_q[x]$. Then $Ord(\alpha_1+ \alpha_2+\dots +\alpha_k) = lcm\{g_1,g_2, \dots , g_k\}$.
    \end{corollary}

    We observe that there are $\frac{q^n-1}{q-1}$ elements of the same norm and $p^{n-1}$ elements of the same trace in a field $\mathbb{F}_{q^n}$. This fact is used to derive the following result.
    \begin{lemma}
        Let $\alpha \in \mathbb{F}_{q^n}$ with $\mathbb{F}_q$-Order $F$ with $(x-1) \nmid F$. Then $trace(\alpha) = 0$.
    \end{lemma}
    \begin{proof}
        $x^n-1=(x-1)(1+x+x^2+\dots+x^{n-1})$. Since, $(x-1) \nmid F, ~F \mid (1+x+x^2+\dots+x^{n-1}) = f$. Thus, by definition of $\mathbb{F}_q$-Order and trace, $trace(\alpha) = 0$.  
    \end{proof}
    \begin{proposition}
        $\displaystyle\sum_{\substack{F \mid x^n-1\\ (x-1)\nmid F}} \phi(F) = p^{n-1}$ with $gcd(n,p)=1$.
    \end{proposition}
    \begin{proof}
        Let $\alpha$ be the element with $\mathbb{F}_q$-Order $F$.\\
        Claim: If $(x-1) \mid F$, then $trace(\alpha) \neq 0$.\\
        Suppose $trace(\alpha)=0$, then $F \mid 1+x+\dots + x^{n-1}$. Since $(x-1) \mid F$, we get $(x-1) \mid 1+x+\dots + x^{n-1}$, a contradiction.($\because gcd(n,p)=1$).\\
        By previous proposition we have if $(x-1) \nmid F$, then $trace(\alpha) =0$. Also there are $p^{n-1}$ elements with same trace in a field $\mathbb{F}_{p^n}$.\\
        Thus, $\displaystyle\sum_{\substack{F \mid x^n-1\\ (x-1)\nmid F}} \phi(F) = p^{n-1}$. 
     \end{proof}

\section{Conclusions}\label{sec5}
This paper explored the theory of $r$-primitive and $k$-normal elements over finite fields and extended these ideas to the polynomial setting through the notion of the $\mathbb{F}_q$-Order. By relating the
$\mathbb{F}_q$-Order of an irreducible polynomial to that of its roots, we reinforced the close connection between element-based and polynomial-based approaches. This perspective presents a unified and systematic framework for studying irreducible polynomials with specified $\mathbb{F}_q$-order, resulting in explicit enumeration outcomes. It produces precise counting formulas for $k$-normal polynomials with a fixed degree. Additionally, order-based and $\mathbb{F}_q$-Order based partitions were analyzed, yielding alternative derivations of classical identities
associated with the Euler phi function and its polynomial counterpart. We also discussed some algebraic and combinatorial properties of
$r$-primitive and $k$-normal elements.

\end{document}